\documentclass[a4paper,12pt]{article}
\usepackage{amsmath,amssymb,theorem}

\newcommand{\frC}{\mathfrak{C}}
\newcommand{\fra}{\mathfrak{a}}
\newcommand{\frb}{\mathfrak{b}}

\newcommand{\frf}{\mathfrak{f}}
\newcommand{\Sine}{\mathcal{S}}
\newcommand{\N}{\mathbb{N}}
\newcommand{\Z}{\mathbb{Z}}
\newcommand{\Q}{\mathbb{Q}}
\newcommand{\R}{\mathbb{R}}
\newcommand{\C}{\mathbb{C}}
\newcommand{\Half}{\mathcal{H}}
\newcommand{\eps}{\varepsilon}
\newcommand{\uprec}{\prec_\mathit{u}}
\newcommand{\lprec}{\prec_\mathit{l}}

\newcommand{\abs}[1]{\lvert #1 \rvert}
\newcommand{\ucl}[1]{\overline{#1}^\mathit{u}}
\newcommand{\lcl}[1]{\overline{#1}^\mathit{l}}

\DeclareMathOperator{\Cl}{\mathit{Cl}}

\DeclareMathOperator{\re}{\mathrm{Re}}

\DeclareMathOperator{\gen}{\mathrm{gen}}
\DeclareMathOperator{\sign}{\mathrm{sign}}
\DeclareMathOperator{\ord}{\mathrm{ord}}

\numberwithin{equation}{section}

\theorembodyfont{\slshape}
\newtheorem{thm}{Theorem}[section]
\newtheorem{prop}[thm]{Proposition}
\newtheorem{cor}[thm]{Corollary}
\newtheorem{lem}[thm]{Lemma}
\newtheorem{defn}[thm]{Definition}
\newtheorem{rem}[thm]{Remark}

\newenvironment{proof}
	{\textbf{Proof.}}
	{\hfill\rule{0.5em}{2ex}\par\bigbreak}

\begin{document}
\title
{On Shintani's ray class invariant for totally real number fields}
\author{Shuji Yamamoto
\thanks{Graduate School of Mathematical Sciences, The University of Tokyo, 
3-8-1 Komaba, Meguro, Tokyo, 153-8914 Japan}}
\date{}
\maketitle

\begin{abstract}
We introduce a ray class invariant $X(\frC)$ for a totally real field, 
following Shintani's work in the real quadratic case. 
We prove a factorization formula $X(\frC)=X_1(\frC)\cdots X_n(\frC)$ 
where each $X_i(\frC)$ corresponds to a real place 
(Theorem \ref{thm:factorization}). 
Although this factorization depends a priori on some choices 
(especially on a cone decomposition), 
we can show that it is actually independent of these choices 
(Theorem \ref{thm:independence}). 
Finally, we describe the behavior of $X_i(\frC)$ 
when the signature of $\frC$ at a real place is changed 
(Theorem \ref{thm:relation}). 
This last result is also interpreted into an interesting behavior 
of the derivative $L'(0,\chi)$ of $L$-functions. 
\end{abstract}

\tableofcontents

\section{Introduction}
Let $F$ be a totally real algebraic number field of degree $n$ and 
$\Cl_F(\frf)$ the narrow ray class group of $F$ 
modulo an integral ideal $\frf$. 
For a technical reason, we assume that $\frf\subsetneq O_F$. 

For a Dirichlet character $\chi\colon\Cl_F(\frf)\to\C^\times$, 
we consider the $L$-function 
\[L(s,\chi)=\sum_{\fra\subset O_F}\chi(\fra)N(\fra)^{-s}
=\sum_{\frC\in\Cl_F(\frf)}\zeta(s,\frC). \]
Here $\zeta(s,\frC)$ denotes the partial zeta function 
associated with a ray class $\frC$. 
Then the leading coefficient 
in the Taylor expansion of $L(s,\chi)$ at $s=0$, 
denoted by $L^*(0,\chi)$, is an important invariant of $F$ and $\chi$, 
especially in the light of the Stark conjecture 
(see e.g.\ \cite{Tate84}). 

In the present paper, we restrict our discussion to the case of order $1$, 
that is, $L^*(0,\chi)=L'(0,\chi)$. 
This condition can be rephrased in terms of the infinite part, 
or the `signature' of $\chi$, as follows. 
Let us number the real places of $F$ and 
denote by $x\mapsto x^{(i)}$ ($i=1,\ldots,n$) 
the corresponding embeddings of $F$ into $\R$. 
Then we choose an element $\mu_i\in F$ for each $i$ such that 
\begin{equation}\label{eq:mu_i}
\mu_i\in 1+\frf,\quad \mu_i^{(i)}<0,\quad \mu_i^{(j)}>0\ (j\ne i) 
\end{equation}
and call the $n$-tuple 
$\bigl(\chi(\mu_1),\ldots,\chi(\mu_n)\bigr)$ of $\pm 1$ 
the \emph{signature} of $\chi$ 
(it is well-defined since the ray class of the principal ideal $(\mu_i)$ is 
independent of the choice of $\mu_i$). 
If $\chi$ is primitive and nontrivial, 
the functional equation for $L(s,\chi)$ tells us that 
the order of vanishing at $s=0$ is equal to the number of $+1$ 
in the signature of $\chi$. 
In particular, in the case of order $1$, 
there is a unique real place for which $\chi(\mu_i)=+1$. 
Hence it is natural to expect that the value $L'(0,\chi)$ may be 
expressed by the contribution of that real place, in some sense. 

To be more precise, we consider the partial zeta functions 
instead of $L$-functions. 
For each ray class $\frC\in\Cl_F(\frf)$, 
we define the \emph{Shintani invariant} $X(\frC)\in\R$ by the formula 
\begin{equation}\label{eq:X(C)}
X(\frC):=
\exp\Bigl(-\zeta'(0,\frC)+(-1)^n\zeta'\bigl(0,\mu\frC\bigr)\Bigr), 
\end{equation}
where $\mu$ is an element of $1+\frf$ which is totally negative 
(for instance, take $\mu=\mu_1\cdots\mu_n$). 
Note that $L'(0,\chi)$ is a linear combination of $\log X(\frC)$: 
\[L'(0,\chi)=-\frac{1}{2}\sum_{\frC\in\Cl_F(\frf)}\chi(\frC)\log X(\frC), \]
since $\chi(\mu)=\chi(\mu_1)\cdots\chi(\mu_n)=(-1)^{n-1}$ in our case. 

Now we explain the main results. 
Let $\Phi$ be a rational cone decomposition of 
the totally positive part of $\R^n=F\otimes_\Q\R$ modulo units 
(the precise definitions of this and the following notations 
will be given in Section 3). 
Choose an integral representative $\fra_0$ of a ray class $\frC$, 
and a set of generators $\gen\sigma=\{\omega_1,\ldots,\omega_d\}$ 
for each cone $\sigma\in\Phi$ from $\fra_0^{-1}\frf$. 
We also denote by $P_\sigma$ the parallelotope spanned by $\gen\sigma$. 
Then we put 
\begin{equation}\label{eq:X_i}
X_i(\frC):=\prod_{\sigma\in\Phi}\prod_{z_\sigma\in P_\sigma\cap(z+\frf)}
\Sine\bigl(z_\sigma^{(i)},(\gen\sigma)^{(i)}\bigr)
\end{equation}
where $\Sine$ is the multiple sine function 
(which will be reviewed in Section 2). 
Our first result (Theorem \ref{thm:factorization}) 
is the following \emph{factorization formula}: 
\begin{equation}\label{eq:Main1}
X(\frC)=\prod_{i=1}^nX_i(\frC). 
\end{equation}

The value $X_i(\frC)$ may be regarded as the contribution of 
the $i$-th real place, but its definition given above 
depends on some auxiliary choices, 
especially on the cone decomposition $\Phi$. 
Our second result, Theorem \ref{thm:independence}, 
states that $X_i(\frC)$ is actually independent of those choices. 

The last main result (Theorem \ref{thm:relation}) is the formula
\begin{equation}\label{eq:Main3}
X_i(\mu_j\frC)=\begin{cases}
X_i(\frC)&(i=j), \\ X_i(\frC)^{-1}&(i\ne j). \end{cases}
\end{equation}
This formula ensures the expected principle that, 
in the value $L'(0,\chi)$, 
only the contribution appears of the unique real place at which $\chi$ has 
positive signature. Indeed, if $i\in\{1,\ldots,n\}$ 
is the unique index such that $\chi(\mu_i)=+1$, we have 
\begin{align*}
&\sum_{\nu_1,\ldots,\nu_n=0}^1
\chi(\mu_1^{\nu_1}\cdots\mu_n^{\nu_n}\frC)
\log X(\mu_1^{\nu_1}\cdots\mu_n^{\nu_n}\frC)\\
&=\chi(\frC)\sum_{\nu_1,\ldots,\nu_n=0}^1(-1)^{\nu_i}
\sum_{j=1}^n\log X_j(\mu_1^{\nu_1}\cdots\mu_n^{\nu_n}\frC)\\
&=\chi(\frC)\sum_{j=1}^n
\sum_{\nu_1,\ldots,\nu_n=0}^1(-1)^{\nu_i}(-1)^{\nu_j}\log X_j(\frC)\\
&=2^n\chi(\frC)\log X_i(\frC), 
\end{align*}
which leads to 
\[L'(0,\chi)=-\frac{1}{2}\sum_{\frC\in\Cl_F(\frf)}\chi(\frC)\log X_i(\frC). \]

\bigskip 
So far, the Shintani invariant has been studied by several authors, 
mostly in the quadratic case $n=2$. 
For example, in that case, the formula \eqref{eq:Main1} was 
proved by Shintani \cite{Shintani77}, and the formula \eqref{eq:Main3} 
was essentially obtained by Arakawa \cite{Arakawa82}. 
Tangedal \cite{Tangedal07} and the author \cite{Yamamoto08} 
also treated the case of $n=2$ by using the theory of continued fractions. 
We remark that all of them considered only some specific cone decompositions. 
In fact, one of the difficulties in treating higher degree is 
involved in controlling the configuration of 
general cone decompositions in $\R^n$. 
In this paper, we overcome it by certain combinatorial discussions, 
especially the notion of `upper and lower closures' introduced in Section 4. 
We also note that the proof given in the present paper is 
not only applicable to higher degree, 
but also much simpler than the previous ones even in the quadratic case 
(see \ref{subsec:n=2}). 

Yoshida \cite{Yoshida03} closely investigated 
the derivatives $\zeta'(0,\frC)$ themselves, 
instead of the combinations $-\zeta'(0,\frC)+(-1)^n\zeta'(0,\frC)$, 
mainly from the viewpoint of the absolute CM periods. 
In particular, he obtained in the case of $n=2$ a result 
\cite[Chapter III, Proposition 6.2]{Yoshida03} 
corresponding to our formula \eqref{eq:Main3}. 
It may be interesting to apply the method of the present paper 
in Yoshida's framework. 

\bigskip 
\emph{Acknowledgement}. 
The author would like to express his gratitude to 
Prof.~T.~Tsuji for valuable discussions and suggestions, 
especially on the product expression \eqref{eq:X_i} of $X_i(\frC)$. 

\subsection{Notation}
The coordinates of a vector $x\in\R^m$ will be denoted by 
$x^{(1)},\ldots,x^{(m)}$. 
We define the \emph{norm} $N\colon\R^m\to\R$ by $N(x)=\prod_{i=1}^mx^{(i)}$. 
For any subset $A$ of $\R^m$, 
we denote by $A_+$ the set of all totally positive elements of $A$, 
i.e., $x\in A$ such that $x^{(i)}>0$ for $i=1,\ldots,m$. 
We also regard $\R^m$ as an $\R$-algebra, 
so that the multiplications are taken componentwise. 

Let $F$ be a totally real algebraic number field of degree $n$. 
For an integral ideal $\frf$ of $F$, 
we denote by $\Cl_F(\frf)$ the narrow ray class group modulo $\frf$, 
and by $E_\frf$ the group of totally positive units 
congruent to $1$ modulo $\frf$.  

We number the real places of $F$ and embed $F$ into $\R^n$. 
It is equivalent to fixing an isomorphism of $\R$-algebras 
$F\otimes_\Q\R\cong\R^n$. 
In particular, for $x\in F$, $N(x)$ is the norm 
with respect to the extension $F/\Q$. 
We also use the notation $N(\fra)$ for the absolute norm of 
a fractional ideal $\fra$ of $F$. 

\smallskip

For $x\in\R$, we define $\langle x\rangle$, 
the \emph{fractional part} of $x$, 
to be the number $t$ such that $x-t\in\Z$ and $0<t\leq 1$. 

\smallskip

In the present paper, a \emph{cone} in $\R^n$ means 
an open simplicial cone, i.e., a subset $\sigma$ of $\R^n$ of the form 
\[\sigma=\{x_1\omega_1+\cdots+x_d\omega_d\mid x_1,\ldots,x_d>0\}, \]
where $\omega_1,\ldots,\omega_d\in\R^n$ are linearly independent. 
The number $d$ of the independent generators is called 
the \emph{dimension} of $\sigma$ and denoted by $d(\sigma)$. 
We regard $\{0\}$ as the only $0$-dimensional cone. 
A cone $\tau$ is called a \emph{face} of $\sigma$ and 
written $\tau\prec\sigma$ 
if it is generated by a subset of $\{\omega_1,\ldots,\omega_d\}$. 

A cone $\sigma$ is called \emph{rational} 
if generators $\omega_1,\ldots,\omega_d$ can be chosen from $F$. 
If $\sigma$ is a rational cone, 
we will always take its generators from $F$. 

\section{Multiple zeta and sine functions}
Here we review definitions and some results about certain zeta functions 
and multiple sine functions. 

In this section, $m$ and $r$ denote natural numbers and 
$\underline{\omega}=(\omega_1,\ldots,\omega_d)$ is 
an $d$-tuple of vectors in $\R^m_+$ (not necessarily linearly independent). 
We consider one more vector $z\in\R^m_+$ of the form 
$z=x_1\omega_1+\cdots+x_d\omega_d$ where $x_1,\ldots,x_d\geq 0$. 

\subsection{Shintani's multiple zeta functions}
\emph{Shintani's multiple zeta function} is defined by 
\begin{equation}\label{eq:zeta(s,z,omega)}
\zeta_{m,d}(s,z,\underline{\omega}):=\sum_{k_1,\ldots,k_d=0}^\infty 
N(z+k_1\omega_1+\cdots+k_d\omega_d)^{-s}. 
\end{equation}
It converges absolutely for $\re(s)>d/m$. 
When $m=1$, $\zeta_{1,d}$ is just the \emph{$d$-ple zeta function of Barnes} 
and denoted by $\zeta_d$. 
Note that, for general $m$ and each $i=1,\ldots,m$, 
we also consider the projection 
$\underline{\omega}^{(i)}=(\omega_1^{(i)},\ldots,\omega_d^{(i)})$ 
and $z^{(i)}$ of the given data, and the associated zeta function 
\[\zeta_d(s,z^{(i)},\underline{\omega}^{(i)})
=\sum_{m_1,\ldots,m_d=0}^\infty 
\bigl(z^{(i)}+m_1\omega_1^{(i)}+\cdots+m_d\omega_d^{(i)}\bigr)^{-s}. \]

It is obvious from the definition that 
the function $\zeta_{m,d}$ satisfies the distribution relation: 
\begin{equation}\label{eq:DistRel}
\begin{split}
&\zeta_{m,d}\bigl(s,z,(\omega_1,\ldots,N\omega_j,\ldots,\omega_d)\bigr)\\
&=\sum_{a=0}^{N-1}\zeta_{m,d}
\bigl(s,z+a\omega_j,(\omega_1,\ldots,\omega_j,\ldots,\omega_d)\bigr)
\end{split}
\end{equation}
for any positive integer $N$. 

These multiple zeta functions are known to be meromorphically continued 
to the whole $s$-plane and holomorphic at $s=0$. 
Moreover, we have the following formulas 
(\cite[Corollary to Proposition 1]{Shintani76} and 
\cite[Proposition 1]{Shintani77b}): 

\begin{prop}\label{prop:s=0}
We have 
\begin{gather}
\zeta_{m,d}(0,z,\underline{\omega})
=\frac{(-1)^d}{m}\sum_{i=1}^m
\sum_{\underline{l}}\prod_{k=1}^d\omega_k^{l_k-1}\frac{B_{l_k}(x_k)}{l_k!}, \\
\zeta_{m,d}'(0,z,\underline{\omega})=
\sum_{i=1}^m\zeta_d'(0,z^{(i)},\underline{\omega}^{(i)})
+\frac{(-1)^d}{m}\sum_{\underline{l}}C_{\underline{l}}(\underline{\omega})
\prod_{k=1}^d\frac{B_{l_k}(x_k)}{l_k!}. 
\end{gather}
Here, $\underline{l}=(l_1,\ldots,l_d)$ runs through $d$-tuples of non-negative 
integers such that $l_1+\cdots+l_d=d$, 
$B_l(x)$ denotes the $l$-th Bernoulli polynomial, and 
\[C_{\underline{l}}(\underline{\omega})=\sum_{i,j\in\{1,\ldots,n\},i\ne j}
\int_0^1\Biggl\{\prod_{k=1}^d(\omega_k^{(i)}+\omega_k^{(j)}u)^{l_k-1}
-\prod_{k=1}^d(\omega_k^{(i)})^{l_k-1}\Biggr\}\frac{du}{u}.\]
\end{prop}

\subsection{The multiple sine functions}
Now we put $\abs{\underline{\omega}}=\omega_1+\cdots+\omega_d$ 
and look at the linear combination 
\begin{equation}\label{eq:xi}
\xi_{m,d}(s,z,\underline{\omega})=-\zeta_{m,d}(s,z,\underline{\omega})
+(-1)^d\zeta_{m,d}(s,\abs{\underline{\omega}}-z,\underline{\omega}). 
\end{equation}
Here we assume that the coefficients $x_1,\ldots,x_d$ of $z$ belong 
to the interval $[0,1]$ and $z\ne 0,\,\abs{\underline{\omega}}$. 

Let us define the function $\Sine_{m,d}(z,\underline{\omega})$ by 
\[\Sine_{m,d}(z,\underline{\omega})
:=\exp\biggl(\frac{\partial}{\partial s}
\xi_{m,d}(s,z,\underline{\omega})\Bigm|_{s=0}\biggr). \]
Again we can apply this definition 
to each projection $\underline{\omega}^{(i)}$ and $z^{(i)}$, 
and then obtain the \emph{multiple sine function} 
$\Sine_d(\underline{\omega}^{(i)},z^{(i)})$ introduced by Kurokawa 
(see \cite{Kurokawa-Koyama03}). 

\begin{prop}
We have the homogeneity 
\begin{equation}\label{eq:Sine_homogeneity}
\Sine_{m,d}(\lambda z,\lambda\underline{\omega})
=\Sine_{m,d}(z,\underline{\omega})
\end{equation}
for any $\lambda\in\R^m_+$, and the factorization formula 
\begin{equation}\label{eq:Sine_factor}
\Sine_{m,d}(z,\underline{\omega})
=\prod_{i=1}^n \Sine_d(z^{(i)},\underline{\omega}^{(i)}). 
\end{equation}
\end{prop}
\begin{proof}
First note that Proposition \ref{prop:s=0} and the property 
\[B_l(1-x)=(-1)^lB_l(x)\]
of the Bernoulli polynomials lead the formulas 
\begin{gather}
\xi_{m,d}(0,z,\underline{\omega})=0, \label{eq:xi(0)}\\
\xi_{m,d}'(0,z,\underline{\omega})
=\sum_{i=1}^n\xi_d'(0,z^{(i)},\underline{\omega}^{(i)}). \label{eq:xi'(0)}
\end{gather}
Hence \eqref{eq:Sine_factor} holds, 
and \eqref{eq:Sine_homogeneity} follows from \eqref{eq:xi(0)} and 
the identity 
\[\xi_{m,d}(s,\lambda z,\lambda\underline{\omega})=
N(\lambda)^{-s}\xi_{m,d}(s,z,\underline{\omega}). \]
\end{proof}

\begin{rem}\upshape 
The function $\Sine_{m,d}(z,\underline{\omega})$, 
considered as a function of $(x_1,\ldots,x_d)$, 
can be meromorphically continued to $\C^d$: 
Indeed, by \eqref{eq:Sine_factor}, it is reduced to 
the meromorphic continuation of each factor 
$\Sine_d(z^{(i)},\underline{\omega}^{(i)})$ as a function of $z^{(i)}$, 
and the latter follows from the continuation of 
the multiple gamma function of Barnes (\cite{Barnes04}). 
\end{rem}

\section{Factorization of the Shintani invariant}
From now on, we fix an integral ideal $\frf\subsetneq O_F$ of $F$ and 
a narrow ray class $\frC\in\Cl_F(\frf)$. 

In this section, we study an expression of the Shintani invariant $X(\frC)$ 
in terms of a certain cone decomposition. 
The main result is Theorem \ref{thm:factorization}. 

\subsection{The Shintani invariant}
Let us choose an integral ideal $\fra_0$ from the class $\frC$. 
Then the partial zeta function associated with $\frC$ can be written as 
\[\zeta(s,\frC)
=\sum_{\alpha\in(1+\fra_0^{-1}\frf)_+/E_\frf}
N\bigl((\alpha)\fra_0\bigr)^{-s}. \]
Moreover, if we also take $z\in F_+$ and put $\frb=z\fra_0^{-1}\frf$, 
then we have 
\begin{equation}\label{eq:zeta(s,C)}
\zeta(s,\frC)=N(\frb^{-1}\frf)^{-s}\,\zeta_\frf(s,z+\frb), 
\end{equation}
where 
\begin{equation}\label{eq:zeta(z+b)}
\zeta_\frf(s,z+\frb):=\sum_{\beta\in(z+\frb)_+/E_\frf}N(\beta)^{-s}. 
\end{equation}
By the assumption $\frf\subsetneq O_F$, we have $z\notin\frb$; 
we will use this fact later. 

Recall that the Shintani invariant $X(\frC)$ is defined by 
\[\log X(\frC)=-\zeta'(0,\frC)+(-1)^n\zeta'(0,\mu\frC), \]
where $\mu$ is a totally negative element of $1+\frf$. 
For the class $\mu\frC$, we may use the data $\fra'_0=\mu\fra_0$ 
and $z'=-\mu z\in F_+$ in places of $\fra_0$ and $z$ for $\frC$. 
Then $z'(\fra'_0)^{-1}\frf=\frb$ and $z'+\frb=-z+\frb$, 
hence we have 
\begin{equation}\label{eq:zeta(s,muC)}
\zeta(s,\mu\frC)=N(\frb^{-1}\frf)^{-s}\,\zeta_\frf(s,-z+\frb). 
\end{equation}
Therefore, we have to study the function 
\begin{equation}\label{eq:xi(z+b)}
\xi_\frf(s,z+\frb):=-\zeta_\frf(s,z+\frb)+(-1)^n\zeta_\frf(s,-z+\frb). 
\end{equation}

Following Shintani \cite{Shintani76,Shintani77b}, we will analyze 
the function $\xi_\frf(s,z+\frb)$ 
by relating it to the multiple zeta functions considered in Section 2. 

In the following, we often omit the subscripts from the notation $\zeta_{m,d}$ 
and simply write as $\zeta(s,z,\underline{\omega})$. 

\subsection{Shintani's cone decomposition}
Shintani investigated the zeta function \eqref{eq:zeta(z+b)}
by using certain cone decomposition of $\R^n_+$. 
Namely, he proved the following result (\cite[Proposition 4]{Shintani76}): 

\begin{thm}
There exists a finite collection $\Phi$ of rational cones in $\R^n_+$ 
such that 
\[\R^n_+=\coprod_{\eps\in E_\frf}\coprod_{\sigma\in\Phi}\eps\sigma. \]
\end{thm}

Now we fix such $\Phi$, 
and set $\tilde{\Phi}=\{\eps\sigma\mid \eps\in E_\frf,\,\sigma\in\Phi\}$. 
By considering an appropriate subdivision, 
we further assume that any face $\ne\{0\}$ of a cone in $\tilde{\Phi}$ 
also belongs to $\tilde{\Phi}$. 
This amounts to assuming that the closures $\overline{\sigma}$ of cones 
$\sigma\in\tilde{\Phi}$ together with the origin $\{0\}$ form a \emph{fan} 
in the sense of toric geometry. 

We remark that the cone decomposition $\tilde{\Phi}$ of $\R^n_+$ is 
locally finite, that is, any compact subset in $\R^n_+$ 
intersects with only finite number of cones in $\tilde{\Phi}$. 
This follows from the finiteness of $\Phi$ and 
the proper discontinuity of the action of $E_\frf$ on $\R^n_+$. 

Recall that a fractional ideal $\frb$ of $F$ and 
an element $z\in F\setminus\frb$ are given. 
For each $\sigma\in\tilde{\Phi}$, 
we choose a set of generators $\gen\sigma=\{\omega_1,\ldots,\omega_d\}$ 
consisting of elements of $\frb$. 
We will often write simply $\sigma$ instead of $\gen\sigma$, 
e.g., $\abs{\sigma}$ means $\abs{\gen\sigma}=\omega_1+\cdots+\omega_d$, 
and $\zeta(s,z,\sigma)=\zeta_{n,d}(s,z,\gen\sigma)$ denotes 
the multiple zeta function. 

By using a fixed set of generators $\gen\sigma=\{\omega_1,\ldots,\omega_d\}$, 
we put 
\[P_\sigma
=\bigl\{x_1\omega_1+\cdots+x_d\omega_d\bigm|0<x_1,\ldots,x_d\leq 1\bigr\}, \]
the parallelotope spanned by $\gen\sigma$. 
Then the set $\sigma$ is decomposed as 
\[\sigma=\coprod_{k_1,\ldots,k_d=0}^\infty 
(P_\sigma+k_1\omega_1+\cdots+k_d\omega_d), \]
from which we deduce the decomposition of the zeta function 
\begin{equation}\label{eq:zeta(s,z+b)}
\zeta_\frf(s,z+\frb)
=\sum_{\sigma\in\Phi}\sum_{\beta\in\sigma\cap(z+\frb)}N(\beta)^{-s}
=\sum_{\sigma\in\Phi}
\sum_{z_\sigma\in P_\sigma\cap(z+\frb)}
\zeta(s,z_\sigma,\sigma). 
\end{equation}
This reduces the study of $\zeta_\frf(s,z+\frb)$ to those of 
Shintani's multiple zeta functions \emph{and} some combinatorics on the cones. 

\subsection{Decomposition formula for $\xi_\frf(s,z+\frb)$}
The formula \eqref{eq:zeta(s,z+b)} seems to lead to an expression 
of the value $\exp\xi_\frf'(0,z+\frb)$ 
as a product of multiple sine functions 
$\Sine(z_\sigma,\sigma)$, where $z_\sigma\in P_\sigma\cap(z+\frb)$. 
There are, however, two apparent differences: 
The definition of $\Sine(z_\sigma,\sigma)$ includes 
the sign $(-1)^{d(\sigma)}$ instead of $(-1)^n$, 
and also vectors $\abs{\sigma}-z_\sigma$ 
which does not necessarily appear in $P_\sigma\cap(-z+\frb)$ 
(because of the boundary of $P_\sigma$). 

We show that, in a sense, these two gaps cancel each other out: 
\begin{prop}\label{prop:combi}
We have 
\begin{equation}\label{eq:combi}
\xi_\frf(s,z+\frb)
=\sum_{\sigma\in\Phi}\sum_{z_\sigma\in P_\sigma\cap(z+\frb)}
\xi(s,z_\sigma,\sigma). 
\end{equation}
\end{prop}
\begin{proof}
By the distribution relation \eqref{eq:DistRel}, 
the right hand side of \eqref{eq:combi} is independent 
of the choice of $\gen\sigma$. 
Hence we may determine $\gen\sigma$ by the condition 
that all elements $\omega\in\gen\sigma$ are primitive in $\frb$, 
i.e.\ $\omega\in\frb$ and 
$\frac{1}{k}\omega\notin\frb$ for any integer $k\geq 2$. 

Now let $\sigma\in\tilde{\Phi}$ and 
$\gen\sigma=\{\omega_1,\ldots,\omega_d\}$. 
We introduce a bijection from $P_\sigma$ onto itself, defined by 
\[z_\sigma=x_1\omega_1+\cdots+x_d\omega_d\longmapsto 
\overline{z_\sigma}:=\langle -x_1\rangle\omega_1+\cdots
+\langle -x_d\rangle\omega_d\]
(recall that $\langle x\rangle\in(0,1]$ denotes 
the fractional part of $x\in\R$). 
It induces a bijection from $P_\sigma\cap(z+\frb)$ onto 
$P_\sigma\cap(-z+\frb)$, hence by \eqref{eq:zeta(s,z+b)} we have 
\begin{equation}\label{eq:combi1}
\xi_\frf(s,z+\frb)
=\sum_{\sigma\in\Phi}\sum_{z_\sigma\in P_\sigma\cap(z+\frb)}
\Bigl\{-\zeta(s,z_\sigma,\sigma)+
(-1)^n\zeta(s,\overline{z_\sigma},\sigma)\Bigr\}. 
\end{equation}
Comparing this with \eqref{eq:combi}, it suffices to show that 
\begin{equation}\label{eq:combi2}
\sum_{\sigma\in\Phi}\sum_{z_\sigma\in P_\sigma\cap(z+\frb)}
(-1)^n\zeta(s,\overline{z_\sigma},\sigma)
=\sum_{\sigma\in\Phi}\sum_{z_\sigma\in P_\sigma\cap(z+\frb)}
(-1)^{d(\sigma)}\zeta(s,\abs{\sigma}-z_\sigma,\sigma). 
\end{equation}

Let us consider the relative interior of $P_\sigma$: 
\[P_\sigma^\circ=
\bigl\{x_1\omega_1+\cdots+x_d\omega_d\bigm|0<x_1,\ldots,x_d<1\bigr\}. \]
Thus $z_\sigma\in P_\sigma$ belongs to $P_\sigma^\circ$ 
if and only if $\overline{z_\sigma}=\abs{\sigma}-z_\sigma$. 
In general, for each $z_\sigma\in P_\sigma$, 
there exists a unique face $\tau\prec\sigma$ and $z_\tau\in P_\tau^\circ$ 
such that 
\[z_\sigma=z_\tau+\abs{\sigma}-\abs{\tau},\qquad 
\overline{z_\sigma}=\abs{\sigma}-z_\tau. \]
Conversely, for $\tau\in\tilde{\Phi}$ and $z_\tau\in P_\tau^\circ$, 
we have $z_\tau+\abs{\sigma}-\abs{\tau}\in P_\sigma$ 
for each $\sigma\in\tilde{\Phi}$ such that $\tau\prec\sigma$. 
Such a pair $(z_\sigma,z_\tau)$ for which $\sigma\in\Phi$ 
can be translated to another pair for which $\tau\in\Phi$ 
by a unique element of $E_\frf$, and vice versa. 
Since each zeta function is invariant under the translation by $E_\frf$, 
we can rewrite each side of \eqref{eq:combi2} as follows: 
\begin{equation}\label{eq:combi4}
\begin{split}
\sum_{\sigma\in\Phi}&\sum_{z_\sigma\in P_\sigma\cap(z+\frb)}
(-1)^n\zeta(s,\overline{z_\sigma},\sigma)\\
&=\sum_{\tau\in\Phi}\sum_{z_\tau\in P_\tau^\circ\cap(z+\frb)}
\sum_{\sigma\in\tilde{\Phi}, \tau\prec\sigma}
(-1)^n\zeta\bigl(s,\abs{\sigma}-z_\tau,\sigma\bigr),
\end{split}
\end{equation}
\begin{equation}\label{eq:combi4'}
\begin{split}
\sum_{\sigma\in\Phi}&\sum_{z_\sigma\in P_\sigma\cap(z+\frb)}
(-1)^{d(\sigma)}\zeta(s,\abs{\sigma}-z_\sigma,\sigma)\\
&=\sum_{\tau\in\Phi}\sum_{z_\tau\in P_\tau^\circ\cap(z+\frb)}
\sum_{\sigma\in\tilde{\Phi}, \tau\prec\sigma}
(-1)^{d(\sigma)}\zeta\bigl(s,\abs{\tau}-z_\tau,\sigma\bigr). 
\end{split}
\end{equation}
Hence it suffices to show the identity 
\begin{equation}\label{eq:combi3}
\sum_{\sigma\in\tilde{\Phi},\tau\prec\sigma}
(-1)^n\zeta\bigl(s,\abs{\sigma}-z_\tau,\sigma\bigr)
=\sum_{\sigma\in\tilde{\Phi}, \tau\prec\sigma}
(-1)^{d(\sigma)}\zeta\bigl(s,\abs{\tau}-z_\tau,\sigma\bigr)
\end{equation}
for each $\tau\in\Phi$ and $z_\tau\in P_\tau^\circ$. 
We need two lemmas: 

\begin{lem}\label{lem:combi1}
If $\tau$ is a face of $\sigma\in\tilde{\Phi}$ and $z_\tau\in P_\tau^\circ$, 
we have 
\[\zeta(s,\abs{\tau}-z_\tau,\sigma)=\sum_{\tau\prec\rho\prec\sigma}
\zeta\bigl(s,\abs{\rho}-z_\tau,\rho\bigr). \]
\end{lem}
\begin{proof}
Put $\gen\tau=\{\omega_1,\ldots,\omega_d\}$ and 
$\gen\sigma=\{\omega_1,\ldots,\omega_r\}$. 
Then the left hand side is the sum of the terms 
$N(-z_\tau+k_1\omega_1+\cdots+k_r\omega_r)^{-s}$ 
where $(k_1,\ldots,k_r)$ runs through $r$-tuples of integers 
such that $k_1,\ldots,k_d\geq 1$ and $k_{d+1},\ldots,k_r\geq 0$. 
For such an $r$-tuple $(k_1,\ldots,k_r)$, let $\rho$ be the cone 
generated by $\omega_j$'s such that $k_j\geq 1$. 
Then the same term $N(-z_\tau+k_1\omega_1+\cdots+k_r\omega_r)^{-s}$ appears 
in the $\rho$-part of the right hand side. 
This makes a bijection of the terms, hence proves the identity. 
\end{proof}

\begin{lem}\label{lem:combi2}
For any $\rho\in\tilde{\Phi}$, we have 
$\sum_{\sigma\in\tilde{\Phi},\rho\prec\sigma}(-1)^{d(\sigma)}=(-1)^n$. 
\end{lem}
\begin{proof}
If we consider the quotient $\R^n/\R\rho\cong\R^{n-d(\rho)}$, 
then the images of $\sigma\in\tilde{\Phi}$ such that $\rho\prec\sigma$ 
together with the point at infinity form a cell decomposition of the sphere 
$S^{n-d(\rho)}=\R^{n-d(\rho)}\cup\{\infty\}$. 
From two expressions of the Euler characteristic 
\[1+\sum_{\sigma}(-1)^{d(\sigma)-d(\rho)}
=\chi(S^{n-d(\sigma)})=1+(-1)^{n-d(\sigma)}, \]
we deduce the lemma. 
\end{proof}

By using these lemmas, we obtain
\begin{align*}
\sum_{\sigma\in\tilde{\Phi},\tau\prec\sigma}
(-1)^{d(\sigma)}\zeta\bigl(s,\abs{\tau}-z_\tau,\sigma\bigr)
&=\sum_{\sigma\in\tilde{\Phi},\tau\prec\sigma}(-1)^{d(\sigma)}
\sum_{\tau\prec\rho\prec\sigma}
\zeta\bigl(s,\abs{\rho}-z_\tau,\rho\bigr)\\
&=\sum_{\rho\in\tilde{\Phi},\tau\prec\rho}
\zeta\bigl(s,\abs{\rho}-z_\tau,\rho\bigr)
\sum_{\sigma\in\tilde{\Phi},\rho\prec\sigma}(-1)^{d(\sigma)}\\
&=\sum_{\rho\in\tilde{\Phi},\tau\prec\rho}
(-1)^n\zeta\bigl(s,\abs{\rho}-z_\tau,\rho\bigr) 
\end{align*}
for $\tau\in\Phi$ and $z_\tau\in P_\tau^\circ$. 
This proves the desired identity \eqref{eq:combi3} and completes the proof 
of Proposition \ref{prop:combi}. 
\end{proof}

\subsection{The factorization of $X(\frC)$}
Now let us prove our first main result: 

\begin{thm}\label{thm:factorization}
If we put 
\begin{equation}\label{eq:X_i(C)}
X_i(\frC)=\prod_{\sigma\in\Phi}\prod_{z_\sigma\in P_\sigma\cap(z+\frb)}
\Sine(z_\sigma^{(i)},\sigma^{(i)}), 
\end{equation}
for $i=1,\ldots,n$, then we have the factorization formula 
\[X(\frC)=\prod_{i=1}^nX_i(\frC). \]
\end{thm}
\begin{proof}
First, Proposition \ref{prop:combi} and \eqref{eq:xi(0)} 
yield that $\xi_\frf(0,z+\frb)=0$ and hence 
\[\log X(\frC)
=\frac{d}{ds}\Bigl(N(\frb^{-1}\frf)^{-s}\,\xi_\frf(s,z+\frb)\Bigr)\Big|_{s=0}
=\xi_\frf'(0,z+\frb). \]
Then the claimed formula follows from 
Proposition \ref{prop:combi} and \eqref{eq:Sine_factor}. 
\end{proof}

Looking at the obvious relation $X(\mu\frC)=X(\frC)^{(-1)^{n-1}}$, 
it is natural to ask whether each factor $X_i(\frC)$ satisfies 
the analogous relation. 
It is indeed the case: 

\begin{prop}\label{eq:X_i(muC)}
We have $X_i(\mu\frC)=X_i(\frC)^{(-1)^{n-1}}$ for each $i=1,\ldots,n$. 
\end{prop}
\begin{proof}
The same argument as the proof of Proposition \ref{prop:combi} 
leads to the identity 
\begin{align*}
\log X_i(\frC)&=\sum_{\sigma\in\Phi}\sum_{z_\sigma\in P_\sigma\cap(z+\frb)}
\xi'(0,z_\sigma^{(i)},\sigma^{(i)})\\
&=\sum_{\tau\in\Phi}\sum_{z_\tau\in P_\tau^\circ\cap(z+\frb)}
\sum_{\sigma\in\tilde{\Phi},\,\tau\prec\sigma}
\xi'\bigl(0,(z_\tau+\abs{\sigma}-\abs{\tau})^{(i)},\sigma^{(i)}\bigr)\\
&=\sum_{\tau\in\Phi}\sum_{z_\tau\in P_\tau^\circ\cap(z+\frb)}
\sum_{\sigma\in\tilde{\Phi},\,\tau\prec\sigma}
(-1)^{d(\sigma)-1}
\xi'\bigl(0,(\abs{\tau}-z_\tau)^{(i)},\sigma^{(i)}\bigr)\\
&=\sum_{\tau\in\Phi}\sum_{z_\tau\in P_\tau^\circ\cap(z+\frb)}
\sum_{\rho\in\tilde{\Phi},\,\tau\prec\rho}
(-1)^{n-1}
\xi'\bigl(0,(\abs{\rho}-z_\tau)^{(i)},\rho^{(i)}\bigr)\\
&=(-1)^{n-1}\sum_{\rho\in\Phi}\sum_{z_\rho\in P_\rho\cap(z+\frb)}
\xi'(0,\overline{z_\rho}^{(i)},\rho^{(i)})\\
&=(-1)^{n-1}\sum_{\rho\in\Phi}\sum_{z_\rho\in P_\rho\cap(-z+\frb)}
\xi'(0,z_\rho^{(i)},\rho^{(i)})\\
&=(-1)^{n-1}\log X_i(\mu\frC). 
\end{align*}
Here the second and the fifth equalities, 
which correspond to \eqref{eq:combi4} and \eqref{eq:combi4'}, 
follow from the homogeneity property \eqref{eq:Sine_homogeneity}. 
\end{proof}

\section{The invariance of $X_i(\frC)$}
We keep the notations in the previous section. 

The definition \eqref{eq:X_i(C)} of $X_i(\frC)$ depends, a priori, 
on the following auxiliary choices. 
\begin{enumerate}
\item an integral representative $\fra_0\in\frC$; 
\item a totally positive number $z\in F_+$; 
\item a finite collection $\Phi$ of rational cones 
explained in Section 3.2; 
\item a set of generators $\gen\sigma$ from $\frb=z\fra_0^{-1}\frf$ 
for each $\sigma\in\Phi$. 
\end{enumerate}
In this section, we will prove that 
the value $X_i(\frC)$ is invariant under any change of these choices. 

\subsection{Preliminary arguments}\label{subsec:InvarPrelim}
It is easy to show 
the independence from $\fra_0$, $z$ and $\gen\sigma$. 
If we replace $z$ by $\lambda z$ for some $\lambda\in F_+$, 
then the definition \eqref{eq:X_i(C)} becomes 
\[\prod_{\sigma\in\Phi}\prod_{z_\sigma\in P_\sigma\cap(z+\frb)}
\Sine\bigl((\lambda z_\sigma)^{(i)},(\lambda\sigma)^{(i)}\bigr)
=\prod_{\sigma\in\Phi}\prod_{z_\sigma\in P_\sigma\cap(z+\frb)}
\Sine(\lambda^{(i)} z_\sigma^{(i)},\lambda^{(i)} \sigma^{(i)}), \]
which is equal to the original one by the homogeneity 
\eqref{eq:Sine_homogeneity} of the multiple sine functions. 
A change of $\fra_0$ amounts to a change of $z$. 
The invariance under a change of $\gen\sigma$ is 
an easy consequence of the distribution relation \eqref{eq:DistRel} 
of the multiple zeta functions. 

To prove the independence from the cone decomposition $\Phi$, 
it suffices to consider only two types of change; 
(1) replacing some $\sigma\in\Phi$ by a translation $\eps\sigma$ 
for some $\eps\in E_\frf$, and 
(2) subdividing some $\sigma\in\Phi$ into a finite sum of rational cones. 
The case (1) can be settled again 
by using the homogeneity: 
\begin{align*}
\prod_{z_{\eps\sigma}\in P_{\eps\sigma}\cap(z+\frb)}
\Sine(z_{\eps\sigma}^{(i)},(\eps\sigma)^{(i)})
&=\prod_{z_\sigma\in P_\sigma\cap(z+\frb)}
\Sine(\eps^{(i)}z_\sigma^{(i)},\eps^{(i)}\sigma^{(i)})\\
&=\prod_{z_\sigma\in P_\sigma\cap(z+\frb)}
\Sine(z_\sigma^{(i)},\sigma^{(i)}). 
\end{align*}
On the other hand, the case (2) is rather difficult, and 
we settle it by introducing the technique of upper and lower closures. 

\subsection{The upper and lower closures}
We fix an index $h\in\{1,\ldots,n\}$, 
and regard the $h$-th coordinate $x^{(h)}$ of a point $x\in\R^n$ 
as the `height' of $x$. We denote by $e_h\in\R^n$ 
the unit vector of the direction of the $h$-th axis, 
i.e.\ $e_h^{(i)}=\delta_{hi}$ (the Kronecker delta). 

\begin{defn}\upshape 
Let $\sigma$ be an $n$-dimensional cone in $\R^n$ and 
$\tau$ a face of it. 
We say that $\tau$ is an \emph{upper face} (resp.\ \emph{lower face}) 
of $\sigma$ and write $\tau\uprec\sigma$ (resp.\ $\tau\lprec\sigma$), 
if there exists $x\in\tau$ such that 
$x-e_h$ (resp.\ $x+e_h$) belongs to $\sigma$.  
\end{defn}

These conditions can be rephrased as follows: 

\begin{prop}\label{prop:UpperCriterion}
Let $\sigma$ be an $n$-dimensional cone 
and $\omega_1,\ldots,\omega_n\in\R^n$ its generators. 
Using the linear expression 
$e_h=a_1\omega_1+\cdots+a_n\omega_n$ 
with respect to the basis 
$\omega_1,\ldots,\omega_n$, put 
\[\Omega_+:=\{\omega_j\mid a_j\geq 0\},\qquad 
\Omega_-:=\{\omega_j\mid a_j\leq 0\}. \] 
Then, for a face $\tau$ of $\sigma$, the following are equivalent: 
\begin{enumerate}
\item $\tau$ is an upper (resp.\ lower) face of $\sigma$. 
\item The cone generated by $\Omega_+$ (resp.\ $\Omega_-$) 
is a face of $\tau$. 
\item For any $x\in\tau$, $x-te_h$ (resp.\ $x+te_h$) 
belongs to $\sigma$ for sufficiently small $t>0$. 
\end{enumerate}
\end{prop}
\begin{proof}
We consider only the upper face conditions. 

Assume that $\tau$ is generated by $\omega_1,\ldots,\omega_d$ 
and let $x=b_1\omega_1+\cdots+b_d\omega_d$ be a point of $\tau$. 
Then, for $t>0$, the point 
\[x-te_h=(b_1-a_1t)\omega_1+\cdots+(b_d-a_dt)\omega_d
-a_{d+1}t\omega_{d+1}-\cdots-a_nt\omega_n\]
belongs to $\sigma$ if and only if the inequalities
$b_j>a_jt$ ($1\leq j\leq d$) and $a_j<0$ ($d+1\leq j\leq n$) hold. 
Since $b_j$ are positive, 
the former $d$ inequalities are always satisfied 
when $t$ is sufficiently small. 
On the other hand, the latter $n-d$ inequalities hold 
if and only if $\{\omega_1,\ldots,\omega_d\}$ contains $\Omega_+$. 
This proves the implications (i)$\implies$(ii)$\implies$(iii), 
while (iii)$\implies$(i) is obvious. 
\end{proof}

The notion of upper and lower faces is particularly useful 
when we treat rational cones 
(with respect to the $\Q$-structure $F\subset\R^n$). 
The basic fact is the following: 

\begin{lem}\label{lem:NonVertical}
For any rational cone $\tau$ of dimension less than $n$, 
the $\R$-subspace generated by $\tau$ does not contain $e_h$. 
\end{lem}
\begin{proof}
Let $V$ denote the $\Q$-subspace of $F$ generated by (rational) 
generators of $\tau$. Since the trace form 
$\langle x,y\rangle=x^{(1)}y^{(1)}+\cdots+x^{(n)}y^{(n)}$ 
on $F$ is non-degenerate, 
there exists a nonzero element $x\in F$ orthogonal to $V$. 
Then, in $\R^n=F\otimes\R$, the inner product 
$\langle x,e_h\rangle=x^{(h)}$ is nonzero, 
which means that $x$ is not orthogonal to $e_h$. 
Hence $e_h$ does not belong to $V\otimes\R$. 
\end{proof}

\begin{prop}\label{prop:ULReciprocity}
For a face $\tau$ of a rational $n$-dimensional cone $\sigma$, we have 
\[\sum_{\tau\prec\rho\uprec\sigma}(-1)^{d(\rho)}=
\begin{cases}(-1)^n&(\tau\lprec\sigma),\\
0&(\text{otherwise}).\end{cases}\]
\end{prop}
\begin{proof}
Let $\gen\sigma=\{\omega_1,\ldots,\omega_n\}$ be 
a set of generators of $\sigma$ and 
$\Omega_\pm$ the subsets defined in Proposition \ref{prop:UpperCriterion}. 
Then Lemma \ref{lem:NonVertical} tells 
that $\gen\sigma=\Omega_+\amalg\Omega_-$. 
Thus, by Proposition \ref{prop:UpperCriterion}, 
$\tau$ is a lower face of $\sigma$ if and only if 
the union of $\gen\tau$ and $\Omega_+$ generates $\sigma$. 

Now let $\rho_0$ be the face generated by $(\gen\tau)\cup\Omega_+$. 
Then, again by Proposition \ref{prop:UpperCriterion}, 
the condition $\tau\prec\rho\uprec\sigma$ is 
equivalent to $\rho_0\prec\rho\prec\sigma$. 
Moreover, if we put $d_0=d(\rho_0)$, 
the number of $d$-dimensional cones $\rho$ 
such that $\rho_0\prec\rho\prec\sigma$ is ${n-d_0 \choose d-d_0}$. 
Thus the binomial theorem leads to the identity 
\[\sum_{\rho_0\prec\rho\prec\sigma}(-1)^{d(\rho)}=
(-1)^{d_0}\cdot\bigl(1+(-1)\bigr)^{n-d_0}=
\begin{cases}(-1)^n&(\rho_0=\sigma),\\
0&(\text{otherwise}).\end{cases}\]
As already mentioned, 
the condition $\rho_0=\sigma$ is equivalent to $\tau\lprec\sigma$, 
hence the proof is complete. 
\end{proof}

\begin{defn}\upshape 
We call the union of all upper faces of $\sigma$ 
(including $\sigma$ itself) 
the \emph{upper closure} of $\sigma$ and denote it by $\ucl{\sigma}$. 
The \emph{lower closure} $\lcl{\sigma}$ is defined similarly. 
\end{defn}

In the following, for any set $\Sigma$ of cones, 
we denote by $\Sigma_d$ the subset consisting of all $d$-dimensional cones 
in $\Sigma$. 

\begin{prop}\label{prop:Phi_n}
Let $\Phi$ and $\tilde{\Phi}$ be sets of rational cones 
considered in Section 3. Then we have 
\[\R^n_+=\coprod_{\sigma\in\tilde{\Phi}_n}\ucl{\sigma}
=\coprod_{\eps\in E_\frf}\coprod_{\sigma\in\Phi_n}\eps\,\ucl{\sigma}. \]
\end{prop}
\begin{proof}
For any point $x\in\R^n_+$, Lemma \ref{lem:NonVertical} implies that 
the vertical line $\{x-te_h\mid t\in\R\}$ intersects 
with each $\tau'\in\tilde{\Phi}\setminus\tilde{\Phi}_n$ 
at at most one point. 
Since the cone decomposition $\tilde{\Phi}$ is locally finite, 
there exists $\delta>0$ such that the segment $\{x-te_h\mid 0<t<\delta\}$ 
lies in a single cone $\sigma\in\tilde{\Phi}_n$, and 
such $\sigma$ is obviously unique. 
This proves the first equality. The second follows from 
the obvious relation $\ucl{\eps\sigma}=\eps\,\ucl{\sigma}$. 
\end{proof}

\begin{prop}\label{prop:subdivision}
Let $\sigma$ be an $n$-dimensional rational cone and 
$\Sigma$ a finite set of rational cones such that 
$\sigma=\coprod_{\tau\in\Sigma}\tau$. Then we have 
\[\ucl{\sigma}=\coprod_{\sigma'\in\Sigma_n}\ucl{\sigma'} \]
and the same for the lower closures. 
\end{prop}
\begin{proof}
Let us denote by $f_A$ the characteristic function of a subset $A\subset\R^n$. 
Then we have $f_\sigma=\sum_{\tau\in\Sigma}f_\tau$ by assumption. 
In particular, 
\[f_\sigma(x)=\sum_{\sigma'\in\Sigma_n}f_{\sigma'}(x)\]
holds whenever $x\notin\coprod_{\tau\in\Sigma\setminus\Sigma_n}\tau$. 
On the other hand, the condition (iii) of 
Proposition \ref{prop:UpperCriterion} shows that 
\[f_{\ucl{\sigma}}(x)=\lim_{t\to +0}f_{\sigma}(x-te_h)\]
for any $x$, and the same for each $\sigma'\in\Sigma_n$. 
Hence we obtain the desired formula 
\[f_{\ucl{\sigma}}(x)=\sum_{\sigma'\in\Sigma_n}f_{\ucl{\sigma'}}(x)\]
by passage to the limit, since the condition 
$x-te_h\notin\coprod_{\tau\in\Sigma\setminus\Sigma_n}\tau$ is satisfied
for any sufficiently small $t>0$ by Lemma \ref{lem:NonVertical}. 
\end{proof}

\subsection{The invariance under a subdivision}
We shall finish the proof of our second main result. 
The final step is the following formula: 

\begin{prop}\label{prop:ULDecomposition}
We have an expression 
\[\log X_i(\frC)
=\sum_{\sigma\in\Phi_n}\frac{d}{ds}
\Biggl(-\sum_{\beta\in\ucl{\sigma}\cap(z+\frb)}(\beta^{(i)})^{-s}
+(-1)^n\sum_{\beta\in\lcl{\sigma}\cap(-z+\frb)}(\beta^{(i)})^{-s}\Biggr)
\Bigg|_{s=0}.\]
\end{prop}
\begin{proof}
Since $X_i(\frC)$ is invariant under the translations of cones in $\Phi$ 
by units, we may assume that 
$\Phi=\{\tau\uprec\sigma\mid\sigma\in\Phi_n\}$ 
by Proposition \ref{prop:Phi_n}. 
Then the definition of $X_i(\frC)$ becomes 
\begin{align*}
\log X_i(\frC)
=-\sum_{\sigma\in\Phi_n}\sum_{\tau\uprec\sigma}
\sum_{z_\tau\in P_\tau\cap(z+\frb)}
\Bigl\{&\zeta'\bigl(0,z_\tau^{(i)},\tau^{(i)}\bigr)\\
&+(-1)^{d(\tau)}\zeta'\bigl(0,(\abs{\tau}-z_\tau)^{(i)},\tau^{(i)}\bigr)
\Bigr\}. 
\end{align*}
For each $\sigma\in\Phi_n$, we have 
\[\sum_{\tau\uprec\sigma}\sum_{z_\tau\in P_\tau\cap(z+\frb)}
\zeta(s,z_\tau^{(i)},\tau^{(i)}\bigr)
=\sum_{\tau\uprec\sigma}\sum_{\beta\in \tau\cap(z+\frb)}(\beta^{(i)})^{-s}
=\sum_{\beta\in\ucl{\sigma}\cap(z+\frb)}(\beta^{(i)})^{-s}, \]
hence the first half is identical to the claimed form. 

On the other hand, for each $\rho\in\tilde{\Phi}$, we can deduce 
\begin{align*}
\sum_{z_\rho\in P_\rho\cap(z+\frb)}
\zeta\bigl(s,(\abs{\rho}-z_\rho)^{(i)},\rho^{(i)}\bigr)
&=\sum_{\beta\in\overline{\rho}\cap(-z+\frb)}(\beta^{(i)})^{-s}\\
&=\sum_{\tau\prec\rho}\sum_{z_\tau\in P_\tau\cap(-z+\frb)}
\zeta(s,z_\tau^{(i)},\tau^{(i)}) 
\end{align*}
from the definition (the usual closure $\overline{\rho}$ is 
the sum of all faces). 
Then, by Proposition \ref{prop:ULReciprocity}, we have 
\begin{align*}
&\sum_{\rho\uprec\sigma}\sum_{z_\rho\in P_\rho\cap(z+\frb)}
(-1)^{d(\rho)}\zeta\bigl(s,(\abs{\rho}-z_\rho)^{(i)},\rho^{(i)}\bigr)\\
&=\sum_{\rho\uprec\sigma}\sum_{\tau\prec\rho}
\sum_{z_\tau\in P_\tau\cap(-z+\frb)}
(-1)^{d(\rho)}\zeta(s,z_\tau^{(i)},\tau^{(i)})\\ 
&=\sum_{\tau\prec\sigma}
\Biggl(\sum_{\tau\prec\rho\uprec\sigma}(-1)^{d(\rho)}\Biggr)
\sum_{z_\tau\in P_\tau\cap(-z+\frb)}\zeta(s,z_\tau^{(i)},\tau^{(i)})\\
&=(-1)^n\sum_{\tau\lprec\sigma}
\sum_{z_\tau\in P_\tau\cap(-z+\frb)}\zeta(s,z_\tau^{(i)},\tau^{(i)}). 
\end{align*}
This means that the second half also satisfies the required equation. 
\end{proof}

\begin{thm}\label{thm:independence}
The value $X_i(\frC)$ is independent of the choices 
of $\fra_0$, $z$, $\Phi$ and $\gen\sigma$. 
\end{thm}
\begin{proof}
As explained in Section \ref{subsec:InvarPrelim}, 
it is sufficient to prove that 
$X_i(\frC)$ is invariant under a subdivision of some $\sigma\in\Phi$. 
By Proposition \ref{prop:subdivision}, such a subdivision does not change 
the sets $\coprod_{\sigma\in\Phi_n}\ucl{\sigma}$ and 
$\coprod_{\sigma\in\Phi_n}\lcl{\sigma}$, 
and Proposition \ref{prop:ULDecomposition} tells that 
$X_i(\frC)$ depends only on these sets. 
This completes the proof.  
\end{proof}

Proposition \ref{prop:ULDecomposition} can be regarded as  
an expression of $X_i(\frC)$ as a product of $n$-ple sine functions. 

\begin{thm}\label{thm:Only_n-pleSine}
For each $\sigma\in\Phi_n$, $\gen\sigma=\{\omega_1,\ldots,\omega_n\}$ be 
the fixed set of generators and define $\Omega_\pm$ as in Proposition 
\ref{prop:UpperCriterion}. Moreover, define another parallelotope 
$P^u_\sigma$ by 
\[P^u_\sigma=\bigl\{x_1\omega_1+\cdots+x_n\omega_n\bigm|
x_j\in(0,1]\ (\omega_j\in\Omega_+),\, 
x_j\in[0,1)\ (\omega_j\in\Omega_-)\bigr\}. \]
Then we have 
\[X_i(\frC)=\prod_{\sigma\in\Phi_n}
\prod_{z_\sigma\in P^u_\sigma\cap(z+\frb)}
\Sine(z_\sigma^{(i)},\sigma^{(i)}). \]
\end{thm}
\begin{proof}
By Proposition \ref{prop:UpperCriterion}, we have 
\[\ucl{\sigma}=\coprod_{k_1,\ldots,k_n=0}^\infty
(P^u_\sigma+k_1\omega_1+\cdots+k_n\omega_n). \]
If we replace $P^u_\sigma$ by $P^l_\sigma$ which is similarly defined, 
\[\lcl{\sigma}=\coprod_{k_1,\ldots,k_n=0}^\infty
(P^l_\sigma+k_1\omega_1+\cdots+k_n\omega_n). \]
Thus the claim follows from Proposition \ref{prop:ULDecomposition} 
and the fact that $z_\sigma\mapsto\abs{\sigma}-z_\sigma$ is 
a bijection from $P^u_\sigma$ onto $P^l_\sigma$. 
\end{proof}

\section{Relation of $X_i(\frC)$ and $X_i(\mu_j\frC)$}
This section is devoted to the proof of the following formula: 

\begin{thm}\label{thm:relation}
The invariants $X_i(\frC)$ satisfies that 
\[X_i(\mu_j\frC)=\begin{cases}
X_i(\frC)& (i=j),\\ X_i(\frC)^{-1}& (i\ne j).\end{cases}\]
\end{thm}
Here $\mu_j$ denotes an element of $1+\frf$ 
such that $\mu_j^{(j)}<0$ and $\mu_j^{(i)}>0$ for $i\ne j$. 

\subsection{Preliminary arguments}\label{subsec:RelPrelim}
First, notice that it suffices to consider only the case $i\ne j$. 
In fact, since the product $\mu=\mu_1\cdots\mu_n$ is 
a totally negative element of $1+\frf$, 
the claimed relations for $i\ne j$ combined with 
Proposition \ref{eq:X_i(muC)} lead to 
\[X_i(\mu_i\frC)=X_i(\mu\frC)^{(-1)^{n-1}}=X_i(\frC). \]
Therefore, from now on, 
we fix mutually distinct indices $i$ and $j$ in $\{1,\ldots,n\}$. 
Since we will use the upper and lower closures again, 
we also fix $h\in\{1,\ldots,n\}$. 
We assume that $h\ne j$, so that the multiplication by $\mu_j$ 
preserves the upper and lower closures. 

Recall that we used the data $\fra_0\in\frC$, $z\in F_+$ and 
$\frb=z\fra_0^{-1}\frf$ to write down 
\[X_i(\frC)=\prod_{\sigma\in\Phi_n}
\prod_{z_\sigma\in P^u_\sigma\cap(z+\frb)}
\Sine(z_\sigma^{(i)},\sigma^{(i)}). \]
Here we use Theorem \ref{thm:Only_n-pleSine} 
instead of the original definition. 
For the class $\mu_j\frC$, we may use the data 
$\mu_j\fra_0$, $z$ and $z(\mu_j\fra_0)^{-1}\frf=\mu_j^{-1}\frb$ 
to obtain 
\[X_i(\mu_j\frC)=\prod_{\sigma\in\Phi_n}
\prod_{z_\sigma\in P^u_\sigma\cap(z+\mu_j^{-1}\frb)}
\Sine(z_\sigma^{(i)},\sigma^{(i)}). \]
Now the assumption $i\ne j$ and the homogeneity of multiple sine functions 
allow the multiplication by the \emph{positive} number $\mu_j^{(i)}$: 
\begin{align*}
X_i(\mu_j\frC)&=\prod_{\sigma\in\Phi_n}
\prod_{z_\sigma\in P^u_\sigma\cap(z+\mu_j^{-1}\frb)}
\Sine(\mu_j^{(i)}z_\sigma^{(i)},\mu_j^{(i)}\sigma^{(i)})\\
&=\prod_{\sigma\in\mu_j\Phi_n}
\prod_{z_\sigma\in P^u_\sigma\cap(\mu_jz+\frb)}
\Sine(z_\sigma^{(i)},\sigma^{(i)})\\
&=\prod_{\sigma\in\mu_j\Phi_n}
\prod_{z_\sigma\in P^u_\sigma\cap(z+\frb)}
\Sine(z_\sigma^{(i)},\sigma^{(i)}). 
\end{align*}
The last equality follows from $\mu_jz+\frb=z+(\mu_j-1)z+\frb=z+\frb$. 
Thus we have to show that the sum
\begin{equation}\label{eq:VanishSum}
\log X_i(\frC)+\log X_i(\mu_j\frC)
=\sum_{\sigma\in\Phi_n\cup(\mu_j\Phi_n)}
\sum_{z_\sigma\in P^u_\sigma\cap(z+\frb)}
\xi'(0,z_\sigma^{(i)},\sigma^{(i)})
\end{equation}
vanishes. 

\subsection{The quadratic case}\label{subsec:n=2}
Here, we deal with the special case of $n=2$ 
to illustrate the idea of our proof. 
In this case, Theorem \ref{thm:relation} was proved in \cite{Yamamoto08} 
by using the continued fraction theory. 
The following method gives another, much simpler proof. 
The reader who is interested only in the general case may skip to 
\ref{subsec:rel_n-dim_cones}. 

Since the unit group $E_\frf$ is of rank $n-1=1$, 
there is a unique generator $\eps$ of $E_\frf$ such that $\eps^{(i)}>1$. 
As a rational cone decomposition, we may take $\Phi=\{\sigma,\tau\}$ 
where $\sigma=\R_+ 1+\R_+\eps$ and $\tau=\R_+1$. 
We also put $\sigma'=\mu_j\sigma$ and $\tau'=\mu_j\tau$ 
(see Figure \ref{fig:n=2}). 
Then the upper and lower closures of $\sigma$ are given by 
$\ucl{\sigma}=\sigma\cup\eps\tau$ and $\lcl{\sigma}=\sigma\cup\tau$ 
(when $n=2$, the conditions $i\ne j$ and $h\ne j$ implies $h=i$). 

\begin{figure}[htb]\caption{Cones in $\eps^k\Phi$ and $\eps^k\mu_j\Phi$}
\label{fig:n=2}
\begin{center}
\setlength{\unitlength}{3pt}
\begin{picture}(120,90)(-60,-10)
\put(-50,0){\vector(1,0){100}} \put(51,-1){$x_j$}
\put(0,-10){\vector(0,1){ 83}} \put(-1,74){$x_i$}
\put(0,0){\line(1,1){40}} \put(41,41){$\tau$}
\put(0,0){\line(1,2){30}} \put(30,61){$\eps\tau$}
\put(0,0){\line(1,4){17}} \put(16,69){$\eps^2\tau$}
\put(20,20){\circle*{1}} \put(21,18){$1$}
\put(14,28){\circle*{1}} \put(15,26){$\eps$}
\put(10,40){\circle*{1}} \put(11,38){$\eps^2$}
\put(32,45){$\sigma$}
\put(20,60){$\eps\sigma$}
\put(10,60){\circle*{0.5}}
\put( 8,63){\circle*{0.5}}
\put( 6,66){\circle*{0.5}}

\put(0,0){\line(-3,2){45}} \put(-49,30){$\tau'$}
\put(0,0){\line(-3,4){40}} \put(-43,54){$\eps\tau'$}
\put(0,0){\line(-1,3){23}} \put(-25,70){$\eps^2\tau'$}
\put(-24,16  ){\circle*{1}} \put(-23,17){$\mu_j$}
\put(-17,22.7){\circle*{1}} \put(-16,24){$\eps\mu_j$}
\put(-12,36  ){\circle*{1}} \put(-11,37){$\eps^2\mu_j$}
\put(-40,37){$\sigma'$}
\put(-30,53){$\eps\sigma'$}
\put(-10,60){\circle*{0.5}}
\put( -8,63){\circle*{0.5}}
\put( -6,66){\circle*{0.5}}
\end{picture}
\end{center}
\end{figure}

By \eqref{eq:VanishSum}, we have to show that the function 
\[A(s)=\sum_{z_\sigma\in P_\sigma^u\cap(z+\frb)}
\xi\bigl(s,z_\sigma^{(i)},\sigma^{(i)}\bigr)
+\sum_{z_{\sigma'}\in P_{\sigma'}^u\cap(z+\frb)}
\xi\bigl(s,z_{\sigma'}^{(i)},\sigma^{\prime(i)}\bigr)\]
has a zero at $s=0$ of order greater than or equal to $2$ 
(one has $A(0)=0$ by \eqref{eq:xi(0)}). 
If $\re(s)$ is large, this function is given by 
\[A(s)=-\sum_{\beta\in(\ucl{\sigma}\cup\ucl{\sigma'})\cap(z+\frb)}
(\beta^{(i)})^{-s}
+(-1)^n\sum_{\beta\in(\lcl{\sigma}\cup\lcl{\sigma'})\cap(-z+\frb)}
(\beta^{(i)})^{-s}. \]

Now we consider similar functions for $\eps^k\sigma$ and $\eps^k\sigma'$ 
($k=0,1,2,\ldots$), and sum up them: 
\[B(s)=\sum_{k=0}^\infty\Biggl\{
-\sum_{\beta\in(\ucl{\eps^k\sigma}\cup\ucl{\eps^k\sigma'})\cap(z+\frb)}
(\beta^{(i)})^{-s}
+(-1)^n\sum_{\beta\in(\lcl{\eps^k\sigma}\cup\lcl{\eps^k\sigma'})\cap(-z+\frb)}
(\beta^{(i)})^{-s}\Biggr\}. \]
It is equivalent to sum up the terms $\bigl((\eps^k\beta)^{(i)}\bigr)^{-s}$ 
for each $\beta\in(\overline{\sigma}^*\cup\overline{\sigma'}^*)\cap(z+\frb)$ 
($*=\mathit{u}$ or $\mathit{l}$), hence one has 
\[B(s)=\frac{1}{1-(\eps^{(i)})^{-s}}A(s). \]
On the other hand, $B(s)$ is the sum over the set 
\begin{equation}\label{eq:cones_n=2}
\bigcup_{k=0}^\infty
(\overline{\eps^k\sigma}^*\cup\overline{\eps^k\sigma'}^*)\cap(z+\frb)
=\overline{(\R_+1+\R_+\mu_j)}^*\cap(z+\frb). 
\end{equation}
If we put $\rho=\R_+1+\R_+\mu_j$, then we have 
\[B(s)=\sum_{z_\rho\in P_\rho^u\cap(z+\frb)}\xi(s,z_\rho^{(i)},\rho^{(i)}). \]
Note that, though $\rho$ is not contained in $\R^n_+$, 
the $\xi$-functions in the right hand side are well-defined 
since the $i$-th projection $\rho^{(i)}$ is positive. 
Thus the function $B(s)$ is a finite sum of $\xi$-functions, 
whose order of zero at $s=0$ is at least $1$ by \eqref{eq:xi(0)}. 
Hence the order of $A(s)=\bigl(1-(\eps^{(i)})^{-s}\bigr)B(s)$ is 
at least $2$, as desired. 
This completes the proof of Theorem \ref{thm:relation} for $n=2$. 

\subsection{Relations for $n$-dimensional cones}\label{subsec:rel_n-dim_cones}
Let us return to the case of general degree $n$. 
In the sequel, all cones we consider are contained in 
the upper half space 
\[\Half:=\{x\in\R^n\mid x^{(h)}>0\}. \]

A key point of the proof for $n=2$ given in \ref{subsec:n=2} is 
the relation \eqref{eq:cones_n=2}, 
which `sums up' an infinite series of cones to a single cone. 
Such a relation will be generalized to higher dimensions, 
as a relation between `$(n-1)$-fold series' and 
`$(n-2)$-fold series' of cones. 
At first, however, we have to prepare some formulas for finite sums of cones. 

For an $n$-tuple $\underline{\omega}=(\omega_1,\ldots,\omega_n)$ 
of vectors in $\Half$, we define the function 
$\chi(\underline{\omega})$ on $\Half$ as follows: 
if $x\in\Half$ can be written as 
$x=a_1\omega_1+\cdots+a_n\omega_n$ for some $a_1,\ldots,a_n>0$, 
we put 
\[\chi(\underline{\omega})(x)
=\sign\det(\omega_1,\ldots,\omega_n), \]
and set $\chi(\underline{\omega})(x)=0$ otherwise. 
Thus, if $\omega_1,\ldots,\omega_n$ are linearly independent, 
$\chi(\underline{\omega})$ is the characteristic function 
(with a sign) of the cone generated by $\underline{\omega}$, 
while $\chi(\underline{\omega})=0$ identically 
in the linearly dependent case. 

A fundamental property of $\chi$ is the `cocycle relation'. 
To state precisely, we need a definition: 
We say that $x\in\Half$ is \emph{generic} 
with respect to a subset $\Omega$ of $\Half$ 
if $x$ does not lie on any cone generated by 
$n-1$ or less elements of $\Omega$. 

\begin{prop}
Let $\omega_0,\ldots,\omega_n\in\Half$ be $n+1$ vectors, 
and assume that $x\in\Half$ is generic 
with respect to $\{\omega_0,\ldots,\omega_n\}$. 
Then we have 
\begin{equation}\label{eq:cocycle}
\sum_{l=0}^n(-1)^l
\chi(\omega_0,\ldots,\check{\omega}_l,\ldots,\omega_n)(x)=0, 
\end{equation}
where $\check{\omega}_l$ means that $\omega_l$ is deleted. 
\end{prop}
\begin{proof}
First note that each value 
$\chi(\omega_0,\ldots,\check{\omega}_l,\ldots,\omega_n)(x)$ 
is invariant when we move $\omega_0,\ldots,\omega_n$ slightly, 
by genericity of $x$. 
Hence we can assume that the points $\omega_0,\ldots,\omega_n\in\R^n$ span 
an $n$-dimensional simplex in $\Half$. 
Then the left hand side of \eqref{eq:cocycle} is exactly 
the intersection number of the boundary of that simplex and the ray $\R_+x$, 
counted with appropriate sign, 
which is zero since the simplex lies in the half space $\Half$ and 
does not contain the origin. 
\end{proof}

Next, we prove a formula (the `prism decomposition') 
for the difference of two functions 
$\chi(\underline{\omega})$ and $\chi(\underline{\eta})$. 
Here we need a definition again: 
For two $(n-1)$-tuples 
$\underline{\omega}=(\omega_1,\ldots,\omega_{n-1})$ and 
$\underline{\eta}=(\eta_1,\ldots,\eta_{n-1})$ of vectors in $\Half$, 
we put 
\[\pi(\underline{\omega},\underline{\eta})
:=\sum_{k=1}^{n-1}(-1)^k
\chi(\omega_1,\ldots,\omega_k,\eta_k,\ldots,\eta_{n-1}). \]

\begin{prop}
Let $\underline{\omega}=(\omega_1,\ldots,\omega_n)$ 
and $\underline{\eta}=(\eta_1,\ldots,\eta_n)$ be 
two $n$-tuples of vectors in $\Half$, 
and $x\in\Half$ be generic with respect to these vectors. 
Then we have 
\begin{equation}\label{eq:prism}
\chi(\underline{\omega})(x)-\chi(\underline{\eta})(x)
=\sum_{l=1}^n(-1)^{l+1}
\pi\bigl(\underline{\omega}[l],\underline{\eta}[l]\bigr)(x). 
\end{equation}
Here $\underline{\omega}[l]$ and $\underline{\eta}[l]$ denote 
$(n-1)$-tuples obtained by deleting $\omega_l$ and $\eta_l$, respectively. 
\end{prop}
\begin{proof}
Here we will drop the argument $x$ from the notation. 

First we rewrite the left hand side as 
\begin{equation}\label{eq:prism1}
\begin{split}
&\chi(\omega_1,\ldots,\omega_n)-\chi(\eta_1,\ldots,\eta_n)\\
&=\sum_{k=1}^n\bigl\{\chi(\omega_1,\ldots,\omega_k,\eta_{k+1},\ldots,\eta_n)
-\chi(\omega_1,\ldots,\omega_{k-1},\eta_k,\ldots,\eta_n)\bigr\}. 
\end{split}
\end{equation}
Then for each $k=1,\ldots,n$, we apply the cocycle relation 
\eqref{eq:cocycle} to the $n+1$ vectors 
$\omega_1,\ldots,\omega_k,\eta_k,\ldots,\eta_n$ 
to obtain 
\[\begin{split}
&\sum_{l=1}^k(-1)^l\chi(\omega_1,\ldots,\check{\omega}_l,\ldots,\omega_k,
\eta_k,\ldots,\eta_n)\\
&+\sum_{l=k}^n(-1)^{l+1}\chi(\omega_1,\ldots,\omega_k,
\eta_k,\ldots,\check{\eta}_l,\ldots,\eta_n)=0, 
\end{split}\]
which amounts to 
\begin{equation}\label{eq:prism2}
\begin{split}
\chi(&\omega_1,\ldots,\omega_k,\eta_{k+1},\ldots,\eta_n)
-\chi(\omega_1,\ldots,\omega_{k-1},\eta_k,\ldots,\eta_n)\\
=&\sum_{l=1}^{k-1}(-1)^{l+k}
\chi(\omega_1,\ldots,\check{\omega}_l,\ldots,\omega_k,\eta_k,\ldots,\eta_n)\\
&+\sum_{l=k+1}^n(-1)^{l+k+1}
\chi(\omega_1,\ldots,\omega_k,\eta_k,\ldots,\check{\eta}_l,\ldots,\eta_n). 
\end{split}
\end{equation}
By substituting \eqref{eq:prism2} to \eqref{eq:prism1}, 
we obtain 
\[\begin{split}
\chi(&\omega_1,\ldots,\omega_n)-\chi(\eta_1,\ldots,\eta_n)\\
=&\sum_{1\leq l<k\leq n}(-1)^{l+k}
\chi(\omega_1,\ldots,\check{\omega}_l,\ldots,\omega_k,\eta_k,\ldots,\eta_n)\\
&+\sum_{1\leq k<l\leq n}(-1)^{l+k+1}
\chi(\omega_1,\ldots,\omega_k,\eta_k,\ldots,\check{\eta}_l,\ldots,\eta_n). 
\end{split}\]
This leads to \eqref{eq:prism}, as easily verified. 
\end{proof}

Next, we extend the generic formula \eqref{eq:prism} 
to all $x\in\Half$ when vectors are rational. 
This can be done by taking the upper or lower closures. 

For $n$-tuples $\underline{\omega}$ and $\underline{\eta}$ 
of vectors in $\Half$, we define 
\begin{gather*}
\ucl{\chi}(\underline{\omega})(x)
=\lim_{t\to +0}\chi(\underline{\omega})(x-te_h), \\
\ucl{\pi}(\underline{\omega},\underline{\eta})(x)
=\lim_{t\to +0}\pi(\underline{\omega},\underline{\eta})(x-te_h). 
\end{gather*}
Similarly, $\lcl{\chi}(\underline{\omega})$ and 
$\lcl{\pi}(\underline{\omega},\underline{\eta})$ 
are defined by replacing $x-te_h$ by $x+te_h$. 

\begin{prop}\label{prop:prism_ucl}
If $\underline{\omega}$ and $\underline{\eta}$ are 
$n$-tuples of vectors in $\Half\cap F$, then 
\begin{equation}\label{eq:prism_ucl}
\ucl{\chi}(\underline{\omega})(x)-\ucl{\chi}(\underline{\eta})(x)
=\sum_{l=1}^n(-1)^{l+1}
\ucl{\pi}\bigl(\underline{\omega}[l],\underline{\eta}[l]\bigr)(x) 
\end{equation}
holds for any $x\in\Half$. 
We also have the same formula for $\lcl{\chi}$ and $\lcl{\pi}$. 
\end{prop}
\begin{proof}
Lemma \ref{lem:NonVertical} implies that 
$x\pm te_h$ is generic for sufficiently small $t>0$. 
Hence the formula is obtained by passage to the limit. 
\end{proof}

Now let us apply the above formula to rational cones. 
To do this, we fix a $d$-tuple (not only a set) of generators 
$\gen\sigma=(\omega_1,\ldots,\omega_d)$ for each rational cone $\sigma$, 
so that each $\omega_k$ is primitive in $\frb$ and 
that $\omega_1^{(h)}>\cdots>\omega_d^{(h)}$. 

When $d(\sigma)=n$, we can regard 
$\gen\sigma$ as an $n\times n$ matrix. 
We denote the sign of the determinant $\det(\gen\sigma)$ 
by $\sign(\sigma)$. Note that $\sign(\sigma)\,\ucl{\chi}(\gen\sigma)$ is 
the characteristic function of the set $\ucl{\sigma}$. 
Moreover, to give an $(n-1)$-dimensional face $\tau\prec\sigma$ 
is equivalent to give a number $l=1,\ldots,n$ such that 
$\gen\tau=(\gen\sigma)[l]$. 
In that situation, we put $\sign(\sigma,\tau)=\sign(\sigma)(-1)^{l+1}$. 

\begin{cor}\label{cor:prism_cone}
If $\eps$ is a totally positive unit, we have 
\begin{equation}
\begin{split}
&\sum_{\sigma\in\Phi_n\cup(\mu_j\Phi_n)}
\sign(\eps\sigma)\ucl{\chi}(\gen\eps\sigma)\\
&=\sum_{\sigma\in\Phi_n}\sum_{\substack{\tau\prec\sigma\\ d(\tau)=n-1}}
\sign(\sigma,\tau)\,\ucl{\pi}(\gen\eps\tau,\gen\mu_j\eps\tau). 
\end{split}
\end{equation}
The same holds for $\lcl{\chi}$ and $\lcl{\pi}$. 
\end{cor}
\begin{proof}
This is a direct consequence of Proposition \ref{prop:prism_ucl}. 
Note that the multiplication by $\eps$ leaves 
$\sign(\sigma)$ and $\sign(\sigma,\tau)$ unchanged, 
while the multiplication by $\mu_j$ changes them. 
\end{proof}

\subsection{The proof of Theorem \ref{thm:relation}}
We shall begin the proof of the relation $X_i(\frC)X_i(\mu_j\frC)=1$. 

Let us choose a basis $\underline{\eps}=(\eps_1,\ldots,\eps_{n-1})$ 
of the unit group $E_\frf$ such that 
$\eps_1^{(i)},\ldots,\eps_{n-1}^{(i)}>1$. 
If $\underline{k}=(k_1,\ldots,k_{n-1})$ is an $(n-1)$-tuple of 
non-negative integers, we write $\underline{\eps}^{\underline{k}}$ 
for the product $\eps_1^{k_1}\cdots\eps_{n-1}^{k_{n-1}}$. 

By \eqref{eq:VanishSum}, it is sufficient to prove that the function 
\[A(s)=\sum_{\sigma\in\Phi_n\cup(\mu_j\Phi_n)}
\sum_{z_\sigma\in P^u_\sigma\cap(z+\frb)}
\xi(s,z_\sigma^{(i)},\sigma^{(i)})\]
has an (at least) double zero at $s=0$, i.e.\ $\ord_{s=0}A(s)\geq 2$. 
We also consider functions similar to $A(s)$ replacing the set 
$\Phi_n\cup(\mu_j\Phi_n)$ by 
$\underline{\eps}^{\underline{k}}\bigl\{\Phi_n\cup(\mu_j\Phi_n)\bigr\}$ 
for all $\underline{k}\in\N^{n-1}$. 
Summing up them, we obtain the function 
\[B(s)=\sum_{\underline{k}=0}^\infty\Biggl\{
\sum_{\sigma\in\Phi_n\cup(\mu_j\Phi_n)}
\sum_{z_\sigma\in P^u_\sigma\cap(z+\frb)}
\xi\bigl(s,(\underline{\eps}^{\underline{k}}z_\sigma)^{(i)},
(\underline{\eps}^{\underline{k}}\sigma)^{(i)}\bigr)\Biggr\}. \]
Then it is easy to see that 
\[B(s)=\frac{1}{1-(\eps_1^{(i)})^{-s}}\cdots
\frac{1}{1-(\eps_{n-1}^{(i)})^{-s}}A(s), \]
hence it suffices to show that $\ord_{s=0}B(s)\geq 2-(n-1)=-(n-3)$. 
We will prove it by expressing $B(s)$ by $(n-2)$-fold infinite sums. 

Corollary \ref{cor:prism_cone} allows us to rewrite $B(s)$ as follows: 
\begin{align}\label{eq:EpsSum_Xi}
&B(s)\notag\\
&=\sum_{\underline{k}=0}^\infty
\sum_{\sigma\in\Phi_n\cup(\mu_j\Phi_n)}
\Biggl\{-\sum_{\beta\in\underline{\eps}^{\underline{k}}\ucl{\sigma}
\cap(z+\frb)}(\beta^{(i)})^{-s}
+(-1)^n\sum_{\beta\in\underline{\eps}^{\underline{k}}\lcl{\sigma}
\cap(-z+\frb)}(\beta^{(i)})^{-s}\Biggr\}\notag\\
&=\sum_{\underline{k}=0}^\infty
\sum_{\sigma\in\Phi_n}\sum_{\substack{\tau\prec\sigma\\ d(\tau)=n-1}}
\sign(\sigma,\tau)\,\Xi(s,\underline{\eps}^{\underline{k}}\tau). 
\end{align}
where $\Xi(s,\tau)$ is defined, 
for each rational $(n-1)$-dimensional cone $\tau$, by 
\[\begin{split}
\Xi(s,\tau)=
&-\sum_{\beta\in(z+\frb)\cap\Half}\ucl{\pi}(\gen\tau,\gen\mu_j\tau)(\beta)
\,(\beta^{(i)})^{-s}\\
&+(-1)^n
\sum_{\beta\in(-z+\frb)\cap\Half}\lcl{\pi}(\gen\tau,\gen\mu_j\tau)(\beta)
\,(\beta^{(i)})^{-s}. 
\end{split}\]

Note that $\Xi(s,\tau)$ is a finite (signed) sum of $\xi$-functions. 
Namely, if $\gen\tau=(\omega_1,\ldots,\omega_{n-1})$, 
we can write 
\[\Xi(s,\tau)=\sum_{k=1}^n(-1)^k
\sum_{z_{\tau_k}\in P_{\tau_k}^u\cap(z+\frb)}
\xi\bigl(s,z_{\tau_k}^{(i)},\tau_k^{(i)}\bigr), \] 
where $\tau_k$ is the cone generated by 
$\omega_1,\ldots,\omega_k,\mu_j\omega_k,\ldots,\mu_j\omega_{n-1}$ 
(if $d(\tau_k)<n$, the $k$-th term should be omitted). 
In particular, we have $\Xi(0,\tau)=0$ by \eqref{eq:xi(0)}. 

The final ingredient is the following simple property of 
$\sign(\sigma,\tau)$: 

\begin{lem}\label{lem:sign(sigma,tau)}
If $\sigma$ and $\sigma'$ are distinct $n$-dimensional rational cones 
and $\tau\prec\sigma,\sigma'$ is a common $(n-1)$-dimensional face, 
then $\sign(\sigma,\tau)=-\sign(\sigma',\tau)$. 
\end{lem}
\begin{proof}
Let $\gen\tau=(\omega_1,\ldots,\omega_{n-1})$, 
and $\gen\sigma=\gen\tau\cup\{\omega\}$ as a set. 
Then the definition of $\sign(\sigma,\tau)$ can be read as 
\[\sign(\sigma,\tau)
=\sign\bigl(\det(\omega,\omega_1,\ldots,\omega_{n-1})\bigr). \]
Similarly, $\sign(\sigma',\tau)
=\sign\bigl(\det(\omega',\omega_1,\ldots,\omega_{n-1})\bigr)$, 
where $\omega'$ is the other generator of $\sigma'$. 
Hence the lemma claims that $\omega$ and $\omega'$ lie 
in opposite sides of the hyperplane spanned by $\omega_1,\ldots,\omega_{n-1}$, 
which is easily verified. 
\end{proof}

Now let us look at the sum \eqref{eq:EpsSum_Xi}. 
If $\sigma\in\Phi_n$ and $\tau\prec\sigma$ is an $(n-1)$-dimensional face, 
there is a unique $\sigma'\in\tilde{\Phi}_n$ distinct from $\sigma$ 
of which $\tau$ is a face. Then we also have a unique $\lambda\in E_\frf$ 
such that $\lambda\sigma'\in\Phi_n$. 
By pairing $(\sigma,\tau)$ and $(\lambda\sigma',\lambda\tau)$ together, 
the corresponding part of the sum \eqref{eq:EpsSum_Xi} becomes 
\begin{align*}
&\sum_{\underline{k}=0}^\infty\Bigl\{
\sign(\sigma,\tau)\Xi(s,\underline{\eps}^{\underline{k}}\tau)
+\sign(\lambda\sigma',\lambda\tau)
\Xi(s,\lambda\underline{\eps}^{\underline{k}}\tau)\Bigr\}\\
&=\sign(\sigma,\tau)\Biggl\{\sum_{\underline{k}=0}^\infty
\Xi(s,\underline{\eps}^{\underline{k}}\tau)
-\sum_{\underline{k}=\underline{m}}^\infty
\Xi(s,\underline{\eps}^{\underline{k}}\tau)\Biggr\}, 
\end{align*}
where $\underline{m}\in\Z^{n-1}$ is determined by 
$\lambda=\underline{\eps}^{\underline{m}}$. 
These two $(n-1)$-fold infinite sums almost cancel each other out, 
the remainder being a finite number of $(n-2)$-fold infinite sums, 
each of which can be written as 
\begin{align*}
\sum_{\underline{k}\in\N^{n-2}}\Xi(s,\underline{\lambda}^{\underline{k}}\tau)
=\frac{1}{1-(\lambda_1^{(i)})^{-s}}\cdots\frac{1}{1-(\lambda_{n-2}^{(i)})^{-s}}
\Xi(s,\tau)
\end{align*}
where $\underline{\lambda}=(\lambda_1,\ldots,\lambda_{n-2})$ is 
an $(n-2)$-tuple of units. 
Hence the sum \eqref{eq:EpsSum_Xi} has a pole of order at most $(n-3)$ 
at $s=0$, and the proof of Theorem \ref{thm:relation} is complete.


\end{document}